\documentclass[]{amsart}

\usepackage{amsthm,amsmath,amssymb,mathrsfs}
\usepackage{mathtools,bbm}
\usepackage{graphicx}
\usepackage{float}

\newtheorem{df}{Definition}
\newtheorem{ass}{Assumption}

\newtheorem{lem}{Lemma}
\newtheorem{thm}{Theorem}
\newtheorem{rmk}{Remark}
\newtheorem{cor}{Corollary}

\title{Solution of linear fractional systems\\with variable coefficients}
\author{Ivan Matychyn and Viktoriia Onyshchenko}
\date{}

\begin{document}
\maketitle
		\begin{abstract}
			The paper deals with the initial value problem for linear systems of FDEs with variable coefficients involving Riemann--Liouville and Caputo derivatives. The technique of the generalized Peano--Baker series is used to obtain the state-transition matrix. Explicit solutions are derived both in the homogeneous and inhomogeneous case. The theoretical results are supported by examples.
		\end{abstract}
	
	\section{Introduction}
	
	Explicit solutions to linear systems of differential equations provide basis to solve control problems. Analytical solutions of the linear systems of fractional differential equations with constant coefficients were derived in the papers \cite{ChikEid,ChikMat1} and then applied to solving control problems and differential games in \cite{Matychyn2015Jun,matychyn2018optimal,matychyn2018time,matychyn2019optimal}.
	
	However only a few papers are devoted to solutions of the systems of FDEs with variable coefficients and their control. In \cite{EckertNagatouRey2019} explicit solutions for the linear systems of initialized \cite{lorenzo2000initialized} FDEs are obtained in terms of generalized Peano--Baker series \cite{baake2011peano}.
	
	This paper deals with the initial value problem for linear systems of FDEs with variable coefficients involving Riemann--Liouville and Caputo derivatives. The technique of the generalized Peano--Baker series is used to obtain the state-transition matrix. Explicit solutions are derived both in the homogeneous and inhomogeneous case. The theoretical results are supported by an example.
	
	\section{Preliminaries}
	
	Denote by $ \mathbb{R}^n $ the $ n$-dimensional Euclidean space and by $ I$ some interval of the real line,  $I\subset\mathbb{R}$. 
	In what follows we will assume that $t_0, t\in I$ and $ t>t_0$.
	Suppose $ f:I \to\mathbb{R}^n$ is an absolutely continuous function. Let us recall that the Riemann--Liouville (left-sided) fractional integral and derivative of order $\alpha$, $0<\alpha<1$, are defined as follows:
	\begin{align*}
	\prescript{}{t_0}J_{t}^\alpha f(t) &= \frac{1}{\Gamma(\alpha)}\int_{t_0}^{t} (t-\tau)^{\alpha-1}f(\tau)d\tau,\\
	\prescript{}{t_0}D_{t}^\alpha f(t) &= \frac{d}{dt} \prescript{}{t_0}J_{t}^{1-\alpha} f(t).
	\end{align*}
	
	The Riemann--Liouville fractional derivative of a constant does not equal zero. Moreover, it becomes infinite as $ t$ approaches $t_0$. That is why the regularized Caputo derivative was introduced, which is free of these shortcomings.
	
	The Caputo (regularized) derivative of fractional order $\alpha$, $0<\alpha<1$, can be introduced by the following formula:
	\begin{equation} \label{caputo}
	\prescript{}{t_0}D_{t}^{(\alpha)} f(t)= J_{a+}^{1-\alpha} \frac{d}{dt} f(t),\ t_0,t\in I.
	\end{equation}

	The following properties of the fractional integrals and derivatives \cite{samko1987integrals,kilbas2006theory,podlubny1998fractional,diethelm2010analysis} will be used in the sequel.
	
	\begin{lem} \label{lem:1}
		If $ \alpha, \beta >0$, and $ f(t)$ is such that the derivatives and integrals below exist, the following equalities hold true:
		\begin{align}
		\prescript{}{t_0}D_{t}^\alpha \prescript{}{t_0}J_{t}^\alpha f(t) &= f(t),\\
		\prescript{}{t_0}D_{t}^{(\alpha)} \prescript{}{t_0}J_{t}^\alpha f(t) &= f(t),\\
		\prescript{}{t_0}J_{t}^\alpha \prescript{}{t_0}J_{t}^\beta f(t) &= \prescript{}{t_0}J_{t}^{\alpha + \beta} f(t) \label{semi}.
		\end{align}
		If, moreover, $ \alpha<1 $, then
		\begin{equation}\label{CapRL}
		\prescript{}{t_0}D_{t}^{(\alpha)} f(t) = \prescript{}{t_0}D_{t}^\alpha f(t) - f(t_0)\frac{(t-t_0)^{-\alpha}}{\Gamma(1-\alpha)}.
		\end{equation}
	\end{lem}
	
	\begin{lem}
		For $ \beta >0$ 
		\begin{align}
		\prescript{}{t_0}J_{t}^\alpha {(t-t_0)^{\beta-1}} &= \frac{\Gamma(\beta)}{\Gamma(\beta+\alpha)} (t-t_0)^{\beta + \alpha -1}, \label{intg0}\\
		\prescript{}{t_0}D_{t}^\alpha {(t-t_0)^{\beta-1}} &= \prescript{}{t_0}D_{t}^{(\alpha)} {(t-t_0)^{\beta-1}} = \frac{\Gamma(\beta)}{\Gamma(\beta-\alpha)} (t-t_0)^{\beta - \alpha -1}, \label{diff0}
		\end{align}
	\end{lem}
	
	In particular, from \eqref{intg0}, \eqref{diff0}, it follows that
	\begin{align}
	\prescript{}{t_0}D_{t}^\alpha \frac{(t-t_0)^{\alpha-1}}{\Gamma(\alpha)} = 0, \label{diff} \\
	\prescript{}{t_0}D_{t}^{(\alpha)} 1 = 0, \label{diff1} \\
	\prescript{}{t_0}J_{t}^{1-\alpha} \frac{(t-t_0)^{\alpha-1}}{\Gamma(\alpha)} = 1. \label{intg}
	\end{align}
	
	Let us formulate some preliminary results on Lebesgue integration before we proceed.
	\begin{thm}[\cite{chen2003introduction}] \label{lebesgue}
		Suppose that $ X, Y \subseteq \mathbb{R} $ are intervals. Suppose also that the function $ f : X \times Y \to \mathbb{R} $
		satisfies the following conditions:
		\begin{enumerate}
			\item[(a)] For every fixed $ y \in Y  $, the function  $ f(\cdot, y) $ is measurable on $ X $.
			\item[(b)] The partial derivative $ \frac{\partial}{\partial y}   f (x, y) $ exists for every interior point $(x, y) \in X \times Y$.
			\item[(c)] There exists a non-negative integrable function $ g $ such that  $ \left| \frac{\partial}{\partial y}   f (x, y) \right| \le g(x) $ for every interior  point $(x, y) \in X \times Y$.
			\item[(d)] There exists $ y_0 \in Y $ such that $ f(x, y_0) $ is integrable on $ X $.
		\end{enumerate}
		Then for every $ y \in Y$, the Lebesgue integral 
		\begin{equation*}
		\int_X f(x,y)dx
		\end{equation*}
		exists. Furthermore, the function $ F : Y \to \mathbb{R} $, defined by
		\begin{equation*}
		F(y) = \int_X f(x,y)dx
		\end{equation*}
		for every $ y \in Y $, is differentiable at every interior point of $ Y $, and the derivative $ F (y)$ satisfies
		\begin{equation*}
		F'(y) = \int_X  \frac{\partial}{\partial y}f(x,y)dx.
		\end{equation*}
	\end{thm}
	\begin{cor} \label{cor:1}
		If $ X=(y_0,y) $,  and hypotheses of Theorem \ref{lebesgue} are fulfilled, the following equality holds true for all $ y\in Y $
		\begin{equation}\label{key}
		\frac{d}{dy} \int_{y_0}^y f(x,y)dx=  \int_{y_0}^y  \frac{\partial}{\partial y}f(x,y)dx + f(y-0,y).
		\end{equation}
	\end{cor}
	This corollary can be obtained by differentiating the function $ G(y,\gamma)= \int_{y_0}^\gamma f(x,y)dx$, then applying Theorem \ref{lebesgue} and the chain rule.

	\section{Homogeneous System of Linear FDEs with Variable Coefficients involving Riemann--Liouville Derivatives}
	Let us consider the following initial value problem:
	\begin{equation} \label{homo}
	\begin{aligned} 
	\prescript{}{t_0}D_{t}^\alpha x(t) &= A(t)x(t),\ t\in \mathring{I},\\
	\prescript{}{t_0}J_{t}^{1-\alpha}x(t)\bigr|_{t=t_0}  &= x_0,
	\end{aligned}
	\end{equation}
	where  the matrix function $ A(t)$ is continuous on $ I$.
	
	\begin{df}
		The state-transition matrix of the system \eqref{homo} is defined as follows:
		\begin{equation}\label{transit}
		\Phi(t,t_0)=\sum_{k=0}^{\infty}\prescript{}{t_0}J_{t}^{k\circ\alpha}A(t),
		\end{equation}
		where 
		\begin{gather*}
		\prescript{}{t_0}J_{t}^{0\circ\alpha}A(t)=\frac{(t-t_0)^{\alpha-1}}{\Gamma(\alpha)}\mathbbm{1},\\
		\prescript{}{t_0}J_{t}^{(k+1)\circ\alpha}A(t)=\prescript{}{t_0}J_{t}^{\alpha}(A(t)\prescript{}{t_0}J_{t}^{k\circ\alpha}A(t)),\ k=0,1,\ldots
		\end{gather*}
		Hereafter $\mathbbm{1}$ stands for an identity matrix.
	\end{df}
	
	We will refer to the series on the right-hand side of \eqref{transit} as the generalized Peano--Baker series \cite{EckertNagatouRey2019,baake2011peano}.
	
	\begin{ass} \label{ass:1}
		The generalized Peano--Baker series in the right-hand side of \eqref{transit} converges uniformly.
	\end{ass}
	

	In view of Lemma \ref{lem:1} and of \eqref{diff}, \eqref{intg}, the following lemma holds true.
	\begin{lem} \label{lem:2}
		Under Assumption \ref{ass:1} the state-transition matrix $ \Phi(t,t_0) $ satisfies the following initial value problem
		\begin{equation*}
		\prescript{}{t_0}D_{t}^\alpha \Phi(t,t_0)= A(t)\Phi(t,t_0),\ \prescript{}{t_0}J_{t}^{1-\alpha}\Phi(t,t_0)\bigr|_{t=t_0}=\mathbbm{1}.
		\end{equation*}
	\end{lem}
	
	Lemma \ref{lem:2} implies the following
	\begin{thm} \label{thm:homo}
		Under Assumption \ref{ass:1} solution to the initial value problem \eqref{homo} is given by the following expression:
		\begin{equation}\label{homosoln}
		x(t)=\Phi(t,t_0)x_0.
		\end{equation}
	\end{thm}
	
	\begin{rmk}
		If $A(t)$ is a constant matrix, i.e. $A(t)\equiv A$, then in view of \eqref{intg0} one gets $$\prescript{}{t_0}J_{t}^{k\circ\alpha}A = \frac{(t-t_0)^{(k+1)\alpha-1}}{\Gamma((k+1)\alpha)}A^k$$ and 
		\begin{equation*}
		\Phi(t,t_0)=e_\alpha^{(t-t_0)A}=t^{\alpha-1}\sum\limits_{k=0}^{\infty} \frac{A^k (t-t_0)^{\alpha k}}{\Gamma[(k+1)\alpha]}=t^{\alpha-1} E_{\alpha,\alpha}(At^\alpha),
		\end{equation*}
		where $ E_{\alpha,\alpha}(At^\alpha) $ is a matrix Mittag-Leffler function and $e_\alpha^{(t-t_0)A}$ is the matrix $\alpha$-exponential function \cite{kilbas2006theory}.
		
		Equation \eqref{homosoln} takes on the form
		\begin{equation*}
		x(t)=e_\alpha^{(t-t_0)A}x_0,
		\end{equation*}
		which is consistent with the formulas, obtained for the systems of fractional differential equations with constant coefficients \cite{kilbas2006theory,Matychyn2015Jun}.
	\end{rmk}
		
	\subsection{Example}
	Let us consider the following system
	\begin{equation}\label{ex:1}
	\begin{gathered}
	\prescript{}{0}D_{t}^\alpha x(t)	=
	\begin{pmatrix}
	0 &t\\
	0 &0
	\end{pmatrix}
	x(t),\\
	\prescript{}{0}J_{t}^{1-\alpha}x(t)\bigr|_{t=0}  = x_0.
	\end{gathered}
	\end{equation}
	Direct calculation yields
	\begin{equation} \label{ex:2RL}
	\Phi(t,0)= \begin{pmatrix}
	\frac{t^{\alpha-1}}{\Gamma(\alpha)} & \frac{\alpha}{\Gamma(2\alpha+1)}t^{2\alpha}\\
	0 & \frac{t^{\alpha-1}}{\Gamma(\alpha)}
	\end{pmatrix}.
	\end{equation}
	It can be readily seen that 
	\begin{align*}
	\prescript{}{0}D_{t}^\alpha \Phi(t,0) &=
	\begin{pmatrix}
	0 & \frac{t^\alpha}{\Gamma(\alpha)}\\
	0 & 0
	\end{pmatrix}
	=A(t)\Phi(t,0),\\
	\prescript{}{0}J_{t}^{1-\alpha}x(t)\bigr|_{t=0} \Phi(t,0) &= 
	\begin{pmatrix}
	1 & \frac{t^{\alpha+1}}{\Gamma(\alpha+1)}\\
	0 & 1
	\end{pmatrix} \Biggr|_{t=0}
	=\mathbbm{1},
	\end{align*}
	hence Lemma \ref{lem:2} holds true.
	
	Suppose that
	\[
	x_0=\begin{pmatrix}1\\ 1\end{pmatrix}.
	\]
	Then, the solution of the initial value problem \eqref{ex:1} can be written down as follows:
	\begin{equation*}
	x(t)=
	\begin{pmatrix}
	\frac{t^{\alpha-1}}{\Gamma(\alpha)} & \frac{\alpha}{\Gamma(2\alpha+1)}t^{2\alpha}\\
	0 & \frac{t^{\alpha-1}}{\Gamma(\alpha)}
	\end{pmatrix}
	\begin{pmatrix}1\\ 1\end{pmatrix} =
		\begin{pmatrix}
	\frac{t^{\alpha-1}}{\Gamma(\alpha)} + \frac{\alpha}{\Gamma(2\alpha+1)}t^{2\alpha}\\
	\frac{t^{\alpha-1}}{\Gamma(\alpha)}
	\end{pmatrix}.
	\end{equation*}

\section{Inhomogeneous system of linear FDEs with variable coefficients involving Riemann--Liouville derivatives}\label{sec:4}

Let us consider partial Riemann-Liouville fractional integral and derivative of order $ \alpha $ ($ 0<\alpha <1$) with respect to $ t$ of a function $ k(t, s) $ of two variables $ (t, s)\in I \times I $, $ k:I\times I \to \mathbb{R} $, defined by
\begin{align}
{_{t_0}^t}J_{t}^\alpha k(t,s) &= \frac{1}{\Gamma(\alpha)}\int_{t_0}^{t} (t-\tau)^{\alpha-1}k(\tau,s)d\tau, \label{part1}\\
{_{t_0}^t}{D}_{t}^\alpha k(t, s) &= \frac{\partial}{\partial t} {_{t_0}^t}J_{t}^{1-\alpha} k(t,s)=  \frac{1}{\Gamma(1-\alpha)} \frac{\partial}{\partial t} \int_{t_0}^{t} (t-\tau)^{-\alpha} k(\tau, s)d\tau. \label{part2}
\end{align}

The following lemma was first formally proposed in \cite{podlubny1998fractional}.
\begin{lem} \label{Podlub}
	Let the function $\varphi(t,s)$, $ \varphi:I\times {I} \to \mathbb{R} $ be such that the following hypotheses are fulfilled
	\begin{enumerate}
		\item[(a)] For every fixed $ t \in I  $, the function  $ \tilde{\varphi}(t,s) = {^t_s}J_{t}^{1-\alpha} \varphi(t,s)$ is measurable on $I $ and integrable on $ I $ w.r.t. $ s $ for some $ t^\ast \in I $.
		\item[(b)] The partial derivative ${^t_s}D_{t}^\alpha {\varphi}(t,s) $ exists for every interior point $ (t, s) \in \mathring{I}\times\mathring{I}$.
		\item[(c)] There exists a non-negative integrable function $ g $ such that  $ \left| {^t_s}D_{t}^\alpha {\varphi}(t,s) \right| \le g(s) $ for every interior  point $(t, s) \in \mathring{I}\times\mathring{I}$.
	\end{enumerate}
	Then
	\begin{equation*}
	{_{t_0}}D_{t}^\alpha \int_{t_0}^t \varphi(t,s)ds = \int_{t_0}^t {^t_s}D_{t}^\alpha \varphi(t,s)ds + \lim\limits_{s \to t-0}{^t_s}J_{t}^{1-\alpha} \varphi(t,s),\ t\in \mathring{I}.
	\end{equation*}
\end{lem}
\begin{proof}
According to \eqref{diff}, we have
\begin{equation*}
{_{t_0}}D_{t}^\alpha \int_{t_0}^t \varphi(t,s)ds = \frac{1}{\Gamma(1-\alpha)} \frac{d}{d t} \int_{t_0}^{t} \frac{d\tau}{(t-\tau)^{-\alpha}}  \int_{t_0}^\tau \varphi(t,s)ds.
\end{equation*}
Now applying consequently Fubini's theorem \cite{chen2003introduction} and Corollary \ref{cor:1}, in view of \eqref{part1}, \eqref{part2}, we obtain
\begin{align*}
{_{t_0}}D_{t}^\alpha \int_{t_0}^t \varphi(t,s)ds &= \frac{1}{\Gamma(1-\alpha)} \frac{d}{d t} \int_{t_0}^{t} \frac{d\tau}{(t-\tau)^{\alpha}}  \int_{t_0}^\tau \varphi(\tau,s)ds\\
&= \frac{1}{\Gamma(1-\alpha)} \frac{d}{d t} \int_{t_0}^{t} ds \int_{s}^{t} \frac{\varphi(\tau,s)d\tau}{(t-\tau)^{\alpha}}\\
&= \frac{d}{d t} \int_{t_0}^{t} {^t_s}J_{t}^{1-\alpha} \varphi(t,s) ds \\
&= \int_{t_0}^t  \frac{\partial}{\partial t}{^t_s}J_{t}^{1-\alpha} \varphi(t,s)ds + \lim\limits_{s\to t-0}{^t_s}J_{t}^{1-\alpha} \varphi(t,s)\\
&= \int_{t_0}^t {^t_s}D_{t}^\alpha {\varphi}(t,s)ds + \lim\limits_{s\to t-0} {^t_s}J_{t}^{1-\alpha} {\varphi}(t,s).
\end{align*}
\end{proof}

Consider the inhomogeneous linear initial value problem
\begin{equation} \label{inhomo}
\begin{aligned} 
{_{t_0}}D_{t}^\alpha x(t) &= A(t)x(t)+u(t),\ t_0,t\in I,\\
{_{t_0}}J_{t}^{1-\alpha}x(t)\bigr|_{t=t_0}  &= x_0
\end{aligned}
\end{equation}
For the sake of simplicity we assume $u:I\to \mathbb{R}^n$ to be continuous on $ I$.

\begin{thm}
	Provided that Assumption 1 is fulfilled, solution to the initial value problem \eqref{inhomo} can be written down as follows:
	\begin{equation} \label{inhomosol}
	x(t) = \Phi(t,t_0)x_0+\int_{t_0}^{t}\Phi(t,\tau) u(\tau) d\tau.	
	\end{equation}
\end{thm}
\begin{proof}
Let us apply the fractional differentiation operator ${_{t_0}}D_{t}^\alpha$ to the both sides of \eqref{inhomosol}. In view of Lemmas \ref{lem:2} and \ref{Podlub} one gets
\begin{equation*}
\begin{aligned}
{_{t_0}}D_{t}^\alpha x(t) &= A(t)\Phi(t,t_0)x_0+\int_{t_0}^{t} {_\tau}D_{t}^\alpha \Phi(t,\tau) u(\tau) d\tau + \lim\limits_{\tau \to t-0}{_\tau}J_{t}^{1-\alpha} \Phi(t,\tau) u(\tau)\\
&=A(t)\Phi(t,t_0)x_0+A(t)\int_{t_0}^{t} \Phi(t,\tau) u(\tau) d\tau + \lim\limits_{\tau \to t-0}{_\tau}J_{t}^{1-\alpha} \Phi(t,\tau) u(\tau)\\
&= A(t)x(t)+\lim\limits_{\tau \to t-0}\left[ \left(\mathbb{I}+ {_\tau}J_{t}^{1-\alpha}\sum_{k=1}^{\infty}{_\tau}J_{t}^{k\circ\alpha}A(t)\right) u(\tau)\right]\\
&= A(t)x(t)+\lim\limits_{\tau \to t-0}\left[ \left(\mathbb{I}+ \sum_{k=1}^{\infty} {_\tau}J_{t}^{1-\alpha} {_\tau}J_{t}^{\alpha} (A(t){_\tau}J_{t}^{(k-1)\circ\alpha}A(t))\right) u(\tau)\right].
\end{aligned}
\end{equation*}
Now taking into account the semigroup property of the fractional integrals \eqref{semi}, one can rewrite the latter expression as follows:
\begin{equation*}
\begin{aligned}
{_{t_0}}D_{t}^\alpha x(t)  
&= A(t)x(t)+\lim\limits_{\tau \to t-0}\left[ \left(\mathbb{I}+ \sum_{k=1}^{\infty} \int_{\tau}^{t} (A(s){_\tau}J_{s}^{(k-1)\circ\alpha}A(s))ds\right) u(\tau)\right]\\
&= A(t)x(t)+u(t),
\end{aligned}
\end{equation*}
which completes the proof.
\end{proof}

\subsection{Example}\label{sec:5}

Let us consider the following system
\begin{equation}\label{ex:1RL}
\begin{gathered}
{_0}D_{t}^\alpha x(t)	=
\begin{pmatrix}
0 &t\\
0 &0
\end{pmatrix}
x(t) + u(t),\\
{_0}J_{t}^{1-\alpha}x(t)\bigr|_{t=0}  = x_0.
\end{gathered}
\end{equation}

Suppose that
\[
x_0=\begin{pmatrix}0\\ 1\end{pmatrix},\quad u(t)\equiv \begin{pmatrix}1\\ 0\end{pmatrix},\ t>0.
\]
Then, in view of \eqref{ex:2RL} the solution of the initial value problem \eqref{ex:1RL} can be written down as follows:
\begin{equation*}
x(t)=
\begin{pmatrix}
\frac{t^{\alpha-1}}{\Gamma(\alpha)} & \frac{\alpha}{\Gamma(2\alpha+1)}t^{2\alpha}\\
0 & \frac{t^{\alpha-1}}{\Gamma(\alpha)}
\end{pmatrix}
\begin{pmatrix}0\\ 1\end{pmatrix} + \int_{0}^{t} 
\begin{pmatrix}
\frac{(t-\tau)^{\alpha-1}}{\Gamma(\alpha)} & \frac{\alpha}{\Gamma(2\alpha+1)}(t-\tau)^{2\alpha}\\
0 & \frac{(t-\tau)^{\alpha-1}}{\Gamma(\alpha)}
\end{pmatrix}
\begin{pmatrix}1\\ 0\end{pmatrix} d\tau.
\end{equation*}
Finally we arrive at the explicit closed-form solution
\begin{align*}
x_1(t) &=\frac{\alpha}{\Gamma(2\alpha+1)}t^{2\alpha} + \frac{t^{\alpha}}{\Gamma(\alpha+1)},\\
x_2(t) &= \frac{t^{\alpha-1}}{\Gamma(\alpha)}.
\end{align*}	

\section{Homogeneous System of Linear FDEs with Variable Coefficients Involving Caputo Derivatives}
We now examine homogeneous linear FDEs with variable coefficients involving Caputo derivatives. Let us consider the following initial value problem:
\begin{equation} \label{homoC}
\begin{aligned} 
\prescript{}{t_0}D_{t}^{(\alpha)} x(t) &= A(t)x(t),\ t\in \mathring{I},\\
x(t_0)  &= \tilde{x}_0,
\end{aligned}
\end{equation}
where  the matrix function $ A(t)$ is continuous on $ I$.

\begin{df}
	The state-transition matrix of the system \eqref{homoC} is defined as follows:
	\begin{equation}\label{transitC}
	\tilde{\Phi}(t,t_0)=\sum_{k=0}^{\infty}\prescript{}{t_0}{\tilde{J}}_{t}^{k\circ\alpha}A(t),
	\end{equation}
	where 
	\begin{gather*}
	\prescript{}{t_0}{\tilde{J}}_{t}^{0\circ\alpha}A(t)=\mathbbm{1},\\
	\prescript{}{t_0}{\tilde{J}}_{t}^{(k+1)\circ\alpha}A(t)=\prescript{}{t_0}J_{t}^{\alpha}(A(t)\prescript{}{t_0}{\tilde{J}}_{t}^{k\circ\alpha}A(t)),\ k=0,1,\ldots
	\end{gather*}
\end{df}

Again, we will refer to the series on the right-hand side of \eqref{transitC} as the generalized Peano--Baker series \cite{EckertNagatouRey2019,baake2011peano}.

\begin{ass} \label{ass:1C}
	The generalized Peano--Baker series in the right-hand side of \eqref{transitC} converges uniformly.
\end{ass}

In view of Lemma \ref{lem:1} and of \eqref{diff}, \eqref{intg}, the following lemma holds true.
\begin{lem} \label{lem:2C}
	Under Assumption \ref{ass:1C} the state-transition matrix $ \tilde{\Phi}(t,t_0) $ satisfies the following initial value problem
	\begin{equation*}
	\prescript{}{t_0}D_{t}^{(\alpha)} \tilde{\Phi}(t,t_0)= A(t)\tilde{\Phi}(t,t_0),\ 
	\tilde{\Phi}(t_0,t_0)=\mathbbm{1}.
	\end{equation*}
\end{lem}

Lemma \ref{lem:2C} implies the following
\begin{thm} \label{thm:homoC}
	Under Assumption \ref{ass:1C} solution to the initial value problem \eqref{homoC} is given by the following expression:
	\begin{equation}\label{homosolnC}
	x(t)=\tilde{\Phi}(t,t_0)\tilde{x}_0.
	\end{equation}
\end{thm}

\begin{rmk}
	If $A(t)$ is a constant matrix, i.e. $A(t)\equiv A$, then in view of \eqref{intg0} one gets $$\prescript{}{t_0}{\tilde{J}}_{t}^{k\circ\alpha}A = \frac{(t-t_0)^{k\alpha}}{\Gamma(k\alpha+1)}A^k$$ and 
	\begin{equation*}
	\tilde{\Phi}(t,t_0)=E_\alpha((t-t_0)^\alpha A)=\sum\limits_{k=0}^{\infty} \frac{A^k (t-t_0)^{\alpha k}}{\Gamma[k\alpha+1]},
	\end{equation*}
	where $E_\alpha(t^\alpha A)= E_{\alpha,1}(t^\alpha A) $.
	
	Equation \eqref{homosolnC} takes on the form
	\begin{equation*}
	x(t)=E_\alpha((t-t_0)^\alpha A)\tilde{x}_0,
	\end{equation*}
	which is consistent with the formulas, obtained for the systems of fractional differential equations with constant coefficients \cite{kilbas2006theory,Matychyn2015Jun}.
\end{rmk}

\subsection{Example}
Let us consider the following system
\begin{equation}\label{ex:2}
\begin{gathered}
\prescript{}{0}D_{t}^{(\alpha)} x(t)	=
\begin{pmatrix}
0 &t\\
0 &0
\end{pmatrix}
x(t),\\
x(0) = \tilde{x}_0.
\end{gathered}
\end{equation}
Direct calculation yields
\begin{equation} \label{ex:2}
\tilde{\Phi}(t,0)= \begin{pmatrix}
1 & \frac{t^{\alpha + 1}}{\Gamma(\alpha+2)}\\
0 & 1
\end{pmatrix}.
\end{equation}
It can be readily seen that 
\begin{align*}
\prescript{}{0}D_{t}^{(\alpha)} \tilde{\Phi}(t,0) &=
\begin{pmatrix}
0 & t\\
0 & 0
\end{pmatrix}
=A(t)\tilde{\Phi}(t,0),\\
\tilde{\Phi}(0,0) &=\mathbbm{1},
\end{align*}
hence Lemma \ref{lem:2C} holds true.

Suppose that
\[
x_0=\begin{pmatrix}1\\ 1\end{pmatrix}.
\]
Then, the solution of the initial value problem \eqref{ex:2} can be written down as follows:
\begin{equation*}
x(t)=
\begin{pmatrix}
1 & \frac{t^{\alpha + 1}}{\Gamma(\alpha+2)}\\
0 & 1
\end{pmatrix}
\begin{pmatrix}1\\ 1\end{pmatrix} =
\begin{pmatrix}
1 + \frac{t^{\alpha + 1}}{\Gamma(\alpha+2)}\\
1
\end{pmatrix}.
\end{equation*}

\section{Inhomogeneous System of Linear FDEs with Variable Coefficients Involving Caputo Derivatives}
Let us first formulate a lemma about Caputo differentiation under the integral sign.


Let us consider partial Caputo derivative of order $ \alpha $ ($ 0<\alpha <1$) with respect to $ t$ of a function $ k(t, s) $ of two variables $ (t, s)\in I \times I $, $ k:I\times I \to \mathbb{R} $, defined by
\begin{equation}
{_{t_0}^t}{D}_{t}^{(\alpha)} k(t, s) =  {_{t_0}^t}J_{t}^{1-\alpha} \frac{\partial}{\partial t} k(t,s)=  \frac{1}{\Gamma(1-\alpha)} \int_{t_0}^{t} (t-\tau)^{-\alpha} \frac{\partial}{\partial \tau} k(\tau, s)d\tau. \label{part2}
\end{equation}

\begin{lem} \label{PodlubC}
	Let the function $\varphi(t,s)$, $ \varphi:I\times {I} \to \mathbb{R} $ satisfy hypotheses of Lemma \ref{Podlub}
	Then
	\begin{equation*}
	{_{t_0}}D_{t}^{(\alpha)} \int_{t_0}^t \varphi(t,s)ds = \int_{t_0}^t {^t_s}D_{t}^{\alpha} \varphi(t,s)ds + \lim\limits_{s \to t-0}{^t_s}J_{t}^{1-\alpha} \varphi(t,s),\ t\in \mathring{I}.
	\end{equation*}
\end{lem}
\begin{proof}
According to \eqref{CapRL}, we have
\begin{equation*}
\begin{aligned}
{_{t_0}}D_{t}^{(\alpha)} \int_{t_0}^t \varphi(t,s)ds &={_{t_0}}D_{t}^{\alpha} \int_{t_0}^t \varphi(t,s)ds - \int_{t_0}^t \varphi(t,s)ds\biggr]_{t=t_0}\frac{(t-t_0)^{-\alpha}}{\Gamma(1-\alpha)}\\
&= {_{t_0}}D_{t}^{\alpha} \int_{t_0}^t \varphi(t,s)ds.
\end{aligned}
\end{equation*}
Thus, the statement of the lemma directly follows from Lemma \ref{Podlub}.
\end{proof}

Consider the inhomogeneous linear initial value problem
\begin{equation} \label{inhomoC}
\begin{aligned} 
{_{t_0}}D_{t}^{(\alpha)} x(t) &= A(t)x(t)+u(t),\ t_0,t\in I,\\
x(t_0) &= x_0
\end{aligned}
\end{equation}
Again, we assume $u:I\to \mathbb{R}^n$ to be continuous on $ I$.

\begin{thm} \label{ThmC}
	Provided that Assumption \ref{ass:1C} is fulfilled, solution to the initial value problem \eqref{inhomoC} can be written down as follows:
	\begin{equation} \label{inhomosolC}
	x(t) = \tilde{\Phi}(t,t_0)x_0+\int_{t_0}^{t} {\Phi}(t,\tau) u(\tau) d\tau.	
	\end{equation}
\end{thm}
\begin{proof}
Let us apply the fractional differentiation operator ${_{t_0}}D_{t}^{(\alpha)}$ to the both sides of \eqref{inhomosolC}. In view of Lemmas \ref{lem:2}, \ref{lem:2C} and \ref{PodlubC} one gets
\begin{equation*}
\begin{aligned}
{_{t_0}}D_{t}^{(\alpha)} x(t) &= A(t)\tilde{\Phi}(t,t_0)x_0+\int_{t_0}^{t} {_\tau}D_{t}^\alpha {\Phi}(t,\tau) u(\tau) d\tau + \lim\limits_{\tau \to t-0}{_\tau}J_{t}^{1-\alpha} {\Phi}(t,\tau) u(\tau)\\
&=A(t)\tilde{\Phi}(t,t_0)x_0+A(t)\int_{t_0}^{t} {\Phi}(t,\tau) u(\tau) d\tau + \lim\limits_{\tau \to t-0}{_\tau}J_{t}^{1-\alpha} {\Phi}(t,\tau) u(\tau)\\
&= A(t)x(t)+\lim\limits_{\tau \to t-0}\left[ \left(\mathbb{I}+ {_\tau}J_{t}^{1-\alpha}\sum_{k=1}^{\infty}{_\tau}{J}_{t}^{k\circ\alpha}A(t)\right) u(\tau)\right]\\
&= A(t)x(t)+\lim\limits_{\tau \to t-0}\left[ \left(\mathbb{I}+ \sum_{k=1}^{\infty} {_\tau}J_{t}^{1-\alpha} {_\tau}J_{t}^{\alpha} (A(t){_\tau}{J}_{t}^{(k-1)\circ\alpha}A(t))\right) u(\tau)\right].
\end{aligned}
\end{equation*}
Now taking into account the semigroup property of the fractional integrals \eqref{semi}, one can rewrite the latter expression as follows:
\begin{equation*}
\begin{aligned}
{_{t_0}}D_{t}^{(\alpha)} x(t)  
&= A(t)x(t)+\lim\limits_{\tau \to t-0}\left[ \left(\mathbb{I}+ \sum_{k=1}^{\infty} \int_{\tau}^{t} (A(s){_\tau}{J}_{s}^{(k-1)\circ\alpha}A(s))ds\right) u(\tau)\right]\\
&= A(t)x(t)+u(t),
\end{aligned}
\end{equation*}
which completes the proof.
\end{proof}

\subsection{Example}\label{sec:5}

Let us consider the following system
\begin{equation}\label{ex:1}
\begin{aligned}
{_0}D_{t}^{(\alpha)} x(t) &=
\begin{pmatrix}
0 &t\\
0 &0
\end{pmatrix}
x(t) + u(t),\\
x(0) &= x_0.
\end{aligned}
\end{equation}

Suppose that
\[
x_0=\begin{pmatrix}0\\ 1\end{pmatrix},\quad u(t)\equiv \begin{pmatrix}1\\ 0\end{pmatrix},\ t>0.
\]
Then, according to Theorem \ref{ThmC} and in view of \eqref{ex:2RL}, \eqref{ex:2}, the solution of the initial value problem \eqref{ex:1} can be written down as follows:
\begin{align*}
x(t) &=
\begin{pmatrix}
1 & \frac{t^{\alpha + 1}}{\Gamma(\alpha+2)}\\
0 & 1
\end{pmatrix}
\begin{pmatrix}0\\ 1\end{pmatrix} + \int_{0}^{t} 
\begin{pmatrix}
\frac{(t-\tau)^{\alpha-1}}{\Gamma(\alpha)} & \frac{\alpha}{\Gamma(2\alpha+1)}(t-\tau)^{2\alpha}\\
0 & \frac{(t-\tau)^{\alpha-1}}{\Gamma(\alpha)}
\end{pmatrix}
\begin{pmatrix}1\\ 0\end{pmatrix} d\tau\\
&=
\begin{pmatrix}
\frac{t^{\alpha+1}}{\Gamma(\alpha+2)}+\frac{t^\alpha}{\Gamma(\alpha+1)}\\
1
\end{pmatrix}.
\end{align*}
Finally we arrive at the explicit closed-form solution
\begin{align*}
x_1(t) &=\frac{t^{\alpha+1}}{\Gamma(\alpha+2)}+\frac{t^\alpha}{\Gamma(\alpha+1)},\\
x_2(t) &= 1.
\end{align*}


\end{document}